\documentclass[12pt,a4paper]{amsart}
\usepackage{amsmath,amsthm,amssymb,amsfonts}

\usepackage[pdftex]{color}
\usepackage[bookmarks=true,hyperindex,pdftex,colorlinks,citecolor=blue,linkcolor=red,urlcolor=orange]
{hyperref}

\newtheorem{theorem}{Theorem}[section]

\newtheorem{proposition}[theorem]{Proposition}
\newtheorem{corollary}[theorem]{Corollary}
\theoremstyle{definition}
\newtheorem{definition}[theorem]{Definition}

\newtheorem{example}[theorem]{Example}

\newtheorem{remark}[theorem]{Remark}

\newcommand{\e}{\varepsilon}

\DeclareMathOperator{\re}{Re}

\DeclareMathOperator{\na}{NA}

\numberwithin{equation}{section}
\newcommand{\R}{\mathbb{R}} %% Conjunto reales:        \R
\newcommand{\C}{\mathbb{C}} %% Conjunto complejos:     \C
\newcommand{\K}{\mathbb{K}} %% Conjunto cuerpo :       \K
\def\K{{\mathbb K}}

\def\R{{\mathbb R}}

\def\C{{\mathbb C}}

\def\re{\hbox{\rm Re~}}

\def\ignora#1{}
\def\n3#1{\left\vert  \! \left\vert \! \left\vert \, #1 \, \right\vert \!
  \right\vert \! \right\vert }
\title[The BPBpp]
{The Bishop-Phelps-Bollob\'as point property}

\author[S. Dantas]{Sheldon Dantas}
\address{Departamento de An\'{a}lisis Matem\'{a}tico, Universidad de Valencia, Doctor Moliner 50, 46100 Burjasot (Valencia), Spain}
\email{sheldon.dantas@uv.es}

\author[S. K. Kim]{Sun Kwang Kim}
\address{Department of Mathematics, Kyonggi University, 443-760 (Suwon), Republic of Korea}
\email{sunkwang@kgu.ac.kr}

\author[H. J. Lee]{Han Ju Lee}
\address{Department of Mathematics Education, Dongguk University, 100-715 (Seoul), Republic of Korea}
\email{hanjulee@dongguk.edu}

\thanks{
The first author were supported by MINECO MTM2014-57838-C2-2-P and CAPES, Doutorado Pleno CSF, BEX 0050/13-0. The second author partially was supported by Basic Science Research Program through the National Research Foundation of Korea(NRF) funded by the Ministry of Education, Science and Technology (2014R1A1A2056084). The third author was supported by the research program of Dongguk University, 2014.}


\date{}


\dedicatory{}

\begin{document}

\begin{abstract}
In this article, we study a version of the Bishop-Phelps-Bollob\'as property. We investigate a pair of Banach spaces $(X, Y)$ such that every operator from $X$ into $Y$ is approximated by operators which attains its norm at the same point where the original operator almost attains its norm. In this case, we say that such a pair has the Bishop-Phelps-Bollob\'as point property (BPBpp). We characterize uniform smoothness in terms of BPBpp and we give some examples of pairs $(X, Y)$ which have and fail this property. Some stability results are obtained about $\ell_1$ and $\ell_\infty$ sums of Banach spaces and we also study this property for bilinear mappings.
 \end{abstract}

\maketitle

\section{Introduction}
\par

One of the most remarkable theorems in the `norm attaining function theory' is the Bishop-Phelps theorem which says that every bounded linear functional can be approximated by norm attaining ones \cite{BP}. At the end of the article of Bishop and Phelps, they asked if it could be true a Bishop-Phelps type theorem for bounded linear operators, but Lindenstrauss \cite{Lind} answered this question in a negative sense. 

 For more information, let us begin with some definitions and notations. Let $X$ be a Banach space over a scalar field $\K$ which is the real field  $\R$ or the  complex field  $\C$. We denote by $S_X$, $B_X$ and $X^*$ the unit sphere, the closed unit ball and the topological dual of $X$, respectively. We say that a bounded linear functional $x^* \in X^*$ {\it attains its norm} if there exists some point $x \in S_X$ such that $|x^*(x)| = \|x^*\|$. In this case, we say that $x^*$ is a {\it norm attaining functional}. We denote by $\na(X)$ the set of all norm attaining functionals. Then the Bishop-Phelps theorem \cite{BP} says that $\overline{\na(X)} = X^*$.  For the set of all bounded linear operators between the Banach spaces $X$ and $Y$, we write $\mathcal{L}(X, Y)$.  In an analogous way, we define a norm attaining operator. Indeed, an operator $T: X \longrightarrow Y$ is called {\it norm attaining} whenever there is $x \in S_X$ such that $\|T(x)\|= \|T\| := \sup_{x \in S_X} \|T(x)\|$, and we use $\na(X, Y)$ for the set of such operators. Note that $X^*=\mathcal{L}(X, \K)$ and $\na(X)=\na(X,\K)$.

In 1970, Bollob\'as \cite{Bol} proved an strengthening of the Bishop-Phelps theorem. He proved that if a functional $x^*$ almost attains its norm at a point $x$, then we can find a  functional $y^*$ which is close to $x^*$ and a point $y$ which is close to $x$ such that $y^*$ attains its norm at $y$. We enunciate this theorem which is nowadays known as the Bishop-Phelps-Bollob\'as theorem.

\begin{theorem} \emph{(Bishop-Phelps-Bollob\'as theorem \cite{Bol}, \cite[Corollary 2.4]{CKMM})} Let $X$ be a Banach space. Let $0<\e <2$ and suppose that $x \in B_X$ and $x^* \in B_{X^*}$ satisfy
\begin{equation*}
\re x^*(x) > 1 - \frac{\e^2}{2}.
\end{equation*}
Then, there are $y \in S_X$ and $y^* \in S_{X^*}$ such that
\begin{equation*}
|y^*(y)| = 1, \ \ \|y - x\| < \e \ \ \mbox{and} \ \ \|y^* - x^*\| < \e.
\end{equation*}
\end{theorem}	
To generalize this theorem to operators, Acosta, Aron, Garc\'ia and Maestre introduced the Bishop-Phelps-Bollob\'as propety. Recall that  a pair of Banach spaces $(X, Y)$ is said to have the {\it Bishop-Phelps-Bollob\'as property} (BPBp, for short) when given $\e > 0$, there exists $\eta(\e) > 0$ such that whenever $T \in \mathcal{L}(X, Y)$ with $\|T\| = 1$ and $x_0 \in S_X$ are such that
\begin{equation*}
\|T(x_0)\| > 1 - \eta(\e),
\end{equation*}
there are $S \in \mathcal{L}(X, Y)$ with $\|S\| = 1$ and $x_1 \in S_X$ such that
\begin{equation*}
\|S(x_1)\| = 1, \ \ \ \|x_1 - x_0\| < \e \ \ \ \mbox{and} \ \ \ \|S - T\| < \e.
\end{equation*}
In this case, we say that the pair $(X, Y)$ has the BPBp with the function $\e \longmapsto \eta(\e)$. In recent years, many interesting problems was solved on this topic and so there is an extensive bibliography about it. For example, we refer the papers \cite{AAGM, ABGM, ACK, ACKLM, CGK, CKMM, CS, Dantas, KL, Martin} for more information and background.

In this article, we study a property for a pair of Banach spaces $(X,Y)$ which insures that it is possible to approximate an operator from $X$ into $Y$ by operators which attains its norm at the same point where the original operator almost attains its norm.
\begin{definition}  A pair of Banach spaces $(X, Y)$ is said to have the {\it Bishop-Phelps-Bollob\'as point property} (BPBpp, for short) if given $\e > 0$, there exists $\eta(\e) > 0$ such that whenever $T \in \mathcal{L}(X, Y)$ with $\|T\| = 1$ and $x \in S_X$ satisfy
\begin{equation*}
\|T(x)\| > 1 - \eta(\e),
\end{equation*}
there exists $S \in \mathcal{L}(X, Y)$ with $\|S\| = 1$ such that
\begin{equation*}
\|S(x)\| = 1 \ \ \ \mbox{and} \ \ \ \|S - T\| < \e.
\end{equation*}
In this case, we say that the pair $(X, Y)$ has the BPBpp with the function $\e \longmapsto \eta(\e)$.
\end{definition}

Note that the BPBpp is stronger than the BPBp by definition. We also observe that in the definition of the BPBpp it is possible to choose $T$ with $\|T\| \leq 1$ instead of $\|T\| = 1$ by changing the parameters. It is worth to mention that a dual property, called the uniform strong Bishop-Phelps-Bollob\'as property (uniform sBPBp),  was already considered \cite{Dantas}.  The difference between that property and the BPBpp is that we perturb the operator not the norm attaining point.

Let $X$, $Y$ and $Z$ be Banach spaces. We denote by $B(X \times Y, Z)$ the set of all bilinear mappings from $X \times Y$ into $Z$. It is a Banach space equipped with the norm
\[ \|B\| = \sup \{ \|B(x, y)\| : x\in B_X, y\in B_Y\}. \]
Recall that the pair $(X \times Y, Z)$ has the {\it Bishop-Phelps-Bollob\'as property for bilinear mappings} (BPBp for bilinear mappings, for short) when given $\e > 0$, there exists $\eta(\e) > 0$ such that whenever $B \in B(X \times Y, Z)$ with $\|B\| = 1$ and $(x_0, y_0) \in S_X \times S_Y$ are such that

\begin{equation*}
\|B(x_0, y_0)\| > 1 - \eta(\e),
\end{equation*}
there are $A \in B(X \times Y, Z)$ with $\|A\| = 1$ and $(x_1, y_1) \in S_X \times S_Y$ such that
\begin{equation*}
\|A(x_1, y_1)\| = 1, \ \|x_1 - x_0\| < \e, \ \|y_1 - y_0\| <\e \ \mbox{and} \ \|A - B\| < \e.
\end{equation*}
It was proved that the pair $(\ell_1 \times \ell_1, \K)$ fails the BPBp for bilinear forms \cite[Theorem 2]{CS} meanwhile the pair $(\ell_1, \ell_{\infty})$ has the BPBp for operators \cite[Theorem 4.1]{AAGM}. On the other hand, if $X$ and $Y$ are uniformly convex Banach spaces, then the pair $(X \times Y, Z)$ has the BPBp for bilinear mappings for any Banach space $Z$ \cite[Theorem 2.2]{ABGM}. For the bilinear case, we introduce the following stronger property.

\begin{definition} We say that a pair of Banach spaces $(X, Y)$ has the {\it Bishop-Phelps-Bollob\'as point property for bilinear mappings} (BPBpp for bilinear mappings, for short) if given $\e > 0$, there exists $\eta(\e) > 0$ such that whenever $B \in B(X \times Y, Z)$ with $\|B\| = 1$ and $(x_0, y_0) \in S_X \times S_Y$ satisfy
	\begin{equation*}
	\|B(x_0, y_0)\| > 1 - \eta(\e),
	\end{equation*}
there exists $A \in B(X \times Y, Z)$ with $\|A\| = 1$ such that
\begin{equation*}
\|A(x_0, y_0)\| = 1 \ \ \ \mbox{and} \ \ \ \|A - B\| < \e.
\end{equation*}
In this case, we say that the pair $(X \times Y, Z)$ has the BPBpp for bilinear mappings with the function $\e \longmapsto \eta(\e)$. When $Z = \K$, we just say that the pair $(X \times Y, \K)$ has the {\it BPBpp for bilinear forms}.
\label{sBPBpbilinear}
\end{definition}
Now we summarize the results in this paper. In section 2, we study the BPBpp for operators. We prove first that the pair $(X, \K)$ has the BPBpp if and only if $X$ is a uniformly smooth Banach space. Also, if we assume that $X$ is a uniformly smooth Banach space and $Y$ has the property $\beta$, then the pair $(X, Y)$ has the property. When $H$ is a Hilbert space,  we show that a pair $(H, Y)$ has the BPBpp for any Banach space $Y$. Another positive result appears when we assume that $X$ is uniformly smooth and $A$ is a uniform algebra. It is proved also that in some cases the property in preserved by direct sums. We finish this section by showing that there exists a $2$-dimensional uniformly smooth real Banach space $X$ such that the pair $(X, Y)$ fails the BPBpp for some Banach space $Y$.

In section 3, we deal with the BPBpp for bilinear mappings. We prove that if $X$ is a uniformly smooth Banach space and $H$ is a Hilbert space then the pair $(X \times H, \K)$ has the BPBpp for bilinear forms if and only if the pair $(X, H^*)$ has the BPBpp for operators. Hence, the pair $(H_1 \times H_2, \K)$ has the BPBpp for bilinear forms for all Hilbert spaces $H_1$ and $H_2$. Also we prove that if $H_1$ and $H_2$ are Hilbert spaces and $Z$ has the property $\beta$, then the pair $(H_1 \times H_2, Z)$ has the property. We finish the paper by showing that the pair $(H_1 \times H_2, C(K))$ has the BPBpp for compact bilinear mappings whenever $H_1$ and $H_2$ are Hilbert spaces and $K$ is a compact Hausdorff topological space. 

\section{The Bishop-Phelps-Bollob\'as point property for operators}

In this section, we study the BPBpp for a pair of Banach spaces $(X, Y)$ on the operator version. First, we deal with the scalar valued case, when $Y = \K$. In this case the property can be rewritten as follows. The pair $(X, \K)$ has BPBpp if and only if for a given $\e > 0$, there exists $\eta(\e) > 0$ such that whenever $x_0^* \in S_{X^*}$ and $x_0 \in S_X$ satisfy $|x_0^*(x_0)| > 1 - \eta(\e)$, there exists $x_1^* \in S_{X^*}$ such that $|x_1^*(x_0)| = 1$ and $\|x_1^* - x_0^*\| < \e$.

 We give a characterization for the uniformly smooth Banach space via the BPBpp as in \cite[Theorem 2.1]{KL} which gives a characterization for uniformly convex Banach spaces via the uniform sBPBp \cite[Definition 3]{Dantas}. A Banach space $X$ is said to be {\it uniformly smooth} if the limit
\begin{equation}
\lim_{t \rightarrow 0} \frac{\|z + tx\| - 1}{t}
\label{us1}
\end{equation}
exists uniformly for all $x \in B_X$ and $z \in S_X$. We recall also that every uniformly smooth Banach space is reflexive and $X$ is uniformly smooth if and only if $X^*$ is uniformly convex. Recall that  a Banach space $X$ is said to be {\it uniformly convex} if given $\e > 0$, there exists a positive real number $\delta(\e) > 0$ such that whenever $x_1, x_2 \in S_X$ are such that $\left\| \frac{x_1 + x_2}{2} \right\| > 1 - \delta(\e)$, then $\|x_1 - x_2\| < \e$. In what follows we use the ideas from \cite[Proposition 4.10]{MMP}. By \cite[Theorem V.9.5, p. 447]{DS} we have that
\begin{equation}
\lim_{t \rightarrow 0^+} \frac{ \|z + tx\| - 1}{t} = \max \{ \re z^*(x): \|z^*\| = z^*(z) = 1 \}.
\label{us2}
\end{equation}

\begin{proposition} The Banach space $X$ is uniformly smooth if and only if the pair $(X, \K)$ has the BPBpp.
\label{DKL1}
\end{proposition}

\begin{proof} Suppose that $X$ is uniformly smooth. Then $X^*$ is uniformly convex. So, given $\e >0$, $x_0^* \in B_{X^*}$ and $x_0 \in S_X$ such that $|x_0^*(x_0)| = |x_0 (x_0^*)| > 1 - \eta(\e)$, there exists $x_1^* \in S_{X^*}$ such that $|x_0(x_1^*)| = |x_1^*(x_0)| = 1$ and $\|x_1^* - x_0^*\| < \e$ \cite[Theorem 2.1]{KL}. This proves that $(X, \K)$ has the BPBpp.

 Conversely, let $\e > 0$ and consider $\eta(\e) > 0$ be the function in the definition of the BPBpp for the pair $(X, \K)$. We prove that the limit (\ref{us1}) exists uniformly for all $x \in B_X$ and $z \in S_X$. Let $x \in B_X$, $z \in S_X$ and $0 < t < \frac{\eta(\e)}{2}$. Define $x_t := \frac{z + tx}{\|z + tx\|} \in S_X$ and take $x_t^* \in S_{X^*}$ to be such that $x_t^*(x_t) = 1$. Since $x_t^*(z) = \|z + tx\| - tx_t^*(x)$, we have that $\re x_t^*(z) > \|z\| - 2t \|x\| > 1 - \eta(\e)$. From the assumption that the pair $(X, \K)$ has the property, there exists $z_t^* \in S_{X^*}$ such that $\re z_t^*(z) = 1$ and $\|z_t^* - x_t^*\| < \e$. Now by the definition of the element $x_t$ and by (\ref{us2}) we have, respectively, that
\begin{equation*}
\frac{\|z + tx\| - 1}{t} = \frac{x_t^*(z + tx) - 1}{t} \leq \re x_t^*(x) \ \ \mbox{and} \ \ \lim_{t \rightarrow 0^+} \frac{\|z + tx\| - 1}{t} \geq \re z_t^*(x).
\end{equation*}
Hence, we have
\begin{equation*}
0 \leq \frac{\|z + tx\| - 1}{t} - \lim_{t \rightarrow 0^+} \frac{\|z + tx\| - 1}{t} \leq \re x_t^*(x) - \re z_t^*(x) \leq \|x_t^* - z_t^*\| < \e.
\end{equation*}
This proves that $X$ is uniformly smooth.
\end{proof}

\begin{remark} We observe that we take $x_0^*$ on $B_{X^*}$ not on $S_{X^*}$ in the proof of Proposition \ref{DKL1} thanks to \cite[Theorem 2.1]{KL}. So we can rewrite this result like this: $X$ is a uniformly smooth Banach space if and only if given $\e > 0$, there is $\eta(\e) > 0$ such that if $x_0^* \in B_{X^*}$ and $x_0 \in S_X$ are such that $|x_0^*(x_0)| > 1 - \eta(\e)$, then there exists $x_1^* \in S_{X^*}$ satisfying $|x_1^*(x_0)| = 1$ and $\|x_1^* - x_0^*\| < \e$.
\end{remark}

As a consequence of the Proposition \ref{DKL1}, we have the following examples.

\begin{itemize}
\item[(a)] If $H$ is a Hilbert space, then the pair $(H, \K)$ has the BPBpp.
\item[(b)] The pair $(L_p(\mu), \K)$ has the BPBpp for a $\sigma$-finite measure $\mu$ and $1 < p < \infty$.
\end{itemize}

Now we start to treat the vector valued case. 

\begin{proposition} Let $X$ be a Banach space. Suppose that there is some Banach space $Y$ such that the pair $(X, Y)$ has the BPBpp. Then $X$ is uniformly smooth.	
\label{DKL2}
\end{proposition}
\begin{proof} Assume that  a pair $(X, Y)$ has BPBpp and for $\e > 0$,  let $\eta(\e) > 0$  be the function in the definition of the BPBpp. We only need to show that the pair $(X, \K)$ has the property because of Proposition \ref{DKL1}. Let $x_0^* \in S_{X^*}$ and $x_0 \in S_X$ be such that
\begin{equation*}	
\re x_0^*(x_0) > 1 - \eta \left( \frac{\e}{2} \right).
\end{equation*}
Define $T: X \longrightarrow Y$ by $T(x) := x_0^*(x)y_0$ for any fixed $y_0 \in S_{Y}$. Since $\|T\| = \|x_0^*\| = 1$ and $\|T(x_0)\| > 1 - \eta(\frac{\e}{2})$, there exists $S \in \mathcal{L}(X, Y)$ with $\|S\| = 1$ such that $\|S(x_0)\| = 1$ and $\|S - T\| < \frac{\e}{2}$. Take $y_0^* \in S_{Y^*}$ so that $\re y_0^* (S(x_0)) = |y_0^*(S(x_0))| = \|S(x_0)\| = 1$ and define $x_1^* := S^*y_0^* \in X^*$. Then we see that  
\begin{equation*}
1\geq \|x_1^*\|\geq \re x_1^*(x_0) = \re S^*y_0^*(x_0) = \re y_0^*(S(x_0)) = 1. 
\end{equation*}
Hence $x_1^* \in S_{X^*}$ and it attains its norm at $x_0$. It remains to prove that $\|x_1^* - x_0^*\| < \e$. By using that $\|S - T\| < \frac{\e}{2}$, we get that
\begin{equation*}
\|x_1^* - y_0^*(y_0)x_0^*\| = \|x_1^* - T^*y_0^*\| = \|S^* y_0^* - T^* y_0^*\| \leq \|S^* - T^*\| < \frac{\e}{2}
\end{equation*}
and since $\re x_1^*(x_0) = 1$,
\begin{equation*}
\re (1 - y_0^*(y_0)) \leq \re (x_1^*(x_0) - y_0^*(y_0)x_0^*(x_0)) \leq \|x_1^* - y_0^*(y_0)x_0^*\| < \frac{\e}{2}.
\end{equation*}
This both inequalities imply that
\begin{equation*}
\|x_1^* - x_0^*\| \leq \|x_1^* - y_0^*(y_0)x_0^*\| + \|y_0^*(y_0)x_0^* - x_0^*\| < \e.
\end{equation*}
This proves that the pair $(X, \K)$ has the BPBpp as desired.
\end{proof}

In particular, all the pairs $(X, Y)$ whenever $X$ is not uniformly smooth, for example $X = c_0$ or $X = \ell_1$, do not have the BPBpp for any Banach space $Y$. Because of that from now on we have to assume that the domain space $X$ is uniformly smooth in order to get more pairs satisfying the property.

In the next result, we prove that for such $X$ whenever $Y$ has the property $\beta$, the pair $(X, Y)$ satisfies the BPBpp. To do so, we use similar arguments to \cite[Theorem 2.2]{AAGM} and \cite[Theorem 4.1]{Martin}. Recall that a Banach space $Y$ is said to  have the {\it property $\beta$ with constant $0 \leq \rho < 1$} if there are sets $\{y_i: i \in I\} \subset S_Y$ and $\{ y_i^*: i \in I \} \subset S_{Y^*}$ such that
\begin{itemize}
\item[(i)] $y_i^*(y_i) = 1$ for all $i \in I$,
\item[(ii)] $|y_i^*(y_j)| \leq \rho < 1$ for all $i, j \in I$ with $i \not= j$,
\item[(iii)] $\|y\| = \sup_{i \in I} |y_i^*(y)|$ for all $y \in Y$.
\end{itemize}
Notice that  $c_0 (I)$ and $\ell_{\infty} (I)$ are the most typical examples with property $\beta$. By \cite[Theorem 2.2]{AAGM} we have that if $Y$ satisfies the property $\beta$ then the pair $(X, Y)$ has the BPBp for any Banach space $X$. We have the analogous result for the BPBpp.

\begin{proposition} Let $X$ and $Y$ be Banach spaces. Assume that $X$ is uniformly smooth and that $Y$ has the property $\beta$. Then the pair $(X, Y)$ has the BPBpp.	
\label{DKL3}
\end{proposition}

\begin{proof} Let $\e > 0$ be given. Suppose that $X$ is uniformly smooth. Proposition \ref{DKL1} says that there exists a positive real number $\eta (\e) > 0$ such that whenever $x_0^* \in B_{X^*}$ and $x_0 \in S_X$ satisfy $|x_0^*(x_0)| > 1 - \eta (\e)$, there is $x_1^* \in S_{X^*}$ such that $|x_1^*(x_0)| = 1$ and  $\|x_1^* - x_0^*\| < \e$. Choose $\xi > 0$ such that
\begin{equation}
1 + \rho\left(\frac{\e}{4} + \xi \right) < \left(1 + \frac{\e}{4} \right)(1 - \xi).
\label{propbeta}
\end{equation}
This gives that $\xi < \frac{\e}{4}$. Let $T \in \mathcal{L}(X, Y)$ with $\|T\| = 1$ and $x_0 \in S_X$ be such that $\|T(x_0)\| > 1 - \eta (\xi)$. Since $Y$ has the property $\beta$, there exists some $\alpha_0 \in \Lambda$ such that $y_{\alpha_0}^*(T(x_0)) = (T^*y_{\alpha_0}^*)(x_0) > 1 - \eta (\xi)$. So there is $x_1^* \in S_{X^*}$ such that $|x_1^*(x_0)| = 1$ and $\|x_1^* - T^*y_{\alpha_0}^*\| < \xi$. Define $S: X \longrightarrow Y$ by
\begin{equation*}
S(x) := T(x) + \left[ \left( 1 + \frac{\e}{4} \right) x_1^*(x) - T^*y_{\alpha_0}^*(x) \right] y_{\alpha_0} \ \ (x \in X).
\end{equation*}
Then $\|S - T\| < \frac{\e}{4} + \xi < \frac{\e}{2}$. Also, we have
\begin{equation*}
S^*y^* = T^*y^* + y^*(y_{\alpha_0}) \left[ \left( 1 + \frac{\e}{4} \right) x_1^* - T^* y_{\alpha_0}^* \right] \ \ \left(y^* \in Y^* \right).
\end{equation*} 
Note that $\|S\| = \sup_{\alpha \in \Lambda} \|S^* y^*\|$. On the one hand, we have that $\|S^* y_{\alpha_0}^*\| = 1 + \frac{\e}{4}$ and on another hand, for $\alpha \not= \alpha_0$, we have that
\begin{equation*}
\|S^* y_{\alpha}^*\| \leq 1 + \rho \left( \frac{\e}{4} + \xi \right) < \left(1 + \frac{\e}{4} \right)(1 - \xi) < 1 + \frac{\e}{4}.
\end{equation*}
This shows that $S^*$ attains its norm at $y_{\alpha_0}^* \in S_{Y^*}$ and consequently $\|S(x_0)\| = \|S\|$. So if $U := \frac{S}{\|S\|}$, then $\|U(x_0)\| = 1$ and $\|U - T\| < 2 \|S - T\| < \e$. Thus the pair $(X, Y)$ has the BPBpp.
\end{proof}

As a consequence of the Proposition \ref{DKL3}, we have the following examples.
\begin{itemize}
\item[(a)] If $H$ is a Hilbert space, then the pairs $(H, c_0)$ and $(H, \ell_{\infty})$ has the BPBpp.
\item[(b)] The pairs $(L_p(\mu), c_0)$ and $(L_p(\mu), \ell_{\infty})$ have the BPBpp for a $\sigma$-finite measure $\mu$ and $1 < p < \infty$.
%%\item[(c)] If $X$ is uniformly smooth and $Y$ is a closed subspace of $\ell_{\infty}$ containing the canonical copy of $c_0$, then $(X, Y)$ has the uniform sBPBP-p.
\end{itemize}
When the domain is a Hilbert space we get a stronger result as it is showed in the next result.

\begin{theorem} Let $H$ be a Hilbert space and let $Y$ be any Banach space. Then the pair $(H, Y)$ has the BPBpp.
\label{DKL4}
\end{theorem}

\begin{proof} Let $H$ be a Hilbert space and let $\e > 0$ be given. Since $H$ is uniformly convex, the pair $(H, Y)$ has the BPBp for all Banach spaces $Y$ (see \cite[Theorem 3.1]{KL} or \cite[Corollary 2.3]{ABGM}). Hence, there exists some function $\e \longmapsto \eta(\e)$ satisfying the BPBp for this pair. Let $T \in \mathcal{L}(H, Y)$ with $\|T\| = 1$ and $h_0 \in S_H$ be such that
\begin{equation*}
\|T(h_0)\| > 1 - \eta\left( \frac{\e}{2} \right).
\end{equation*}
Then there are  $\widetilde{S} \in \mathcal{L}(H, Y)$ with $\|\widetilde{S}\| = 1$ and $\widetilde{h}_0 \in S_H$ satisfying that
\begin{equation*}
\|\widetilde{S}(\widetilde{h}_0)\| = 1, \ \ \ \|\widetilde{S} - T\| < \frac{\e}{2} \ \ \ \mbox{and} \ \ \ \|h_0 - \widetilde{h}_0\| < \frac{\e}{2}.
\end{equation*}
Since $H$ is Hilbert, there is a linear isometry $R: H \longrightarrow H$ with $\|R\| = 1$ such that
\begin{equation*}
R(h_0) = \widetilde{h}_0 \ \ \ \mbox{and} \ \ \ \|R - Id_H\| < \frac{\e}{2}.
\end{equation*}
Define $S := \widetilde{S} \circ R: H \longrightarrow Y$. Then $\|S\| \leq 1$ and
\begin{equation*}
\|S(h_0)\| = \| \widetilde{S}(R(h_0))\| = \| \widetilde{S} (\widetilde{h}_0)\| = 1.
\end{equation*}
So $\|S\| = \|S(h_0)\| = 1$. Moreover,
\begin{equation*}
\|S - T\| \leq \| \widetilde{S} \circ R - \widetilde{S}\| + \| \widetilde{S} - T\| \leq \| R - Id_H\| + \frac{\e}{2} < \e.
\end{equation*}
This proves that the pair $(H, Y)$ has the BPBpp as desired.
\end{proof}

%%As a consequence of the Theorem \ref{DKL4}, we have the follow examples.
%%\begin{itemize}
%%\item[(a)] If $H$ is a Hilbert space and $Y$ is a finite-dimensional Banach space, then the pair $(X, Y)$ has the BPBpp.
%%\item[(b)] The pair $(L_2[a, b], L_p(\mu))$ has the uniform sBPB-p for any measure $\mu$ and $1 \leq p \leq \infty$.
%%\end{itemize}
%%\vspace{0.2cm}

Let $K$ be a compact Hausdorff topological space. We denote by $C(K)$ the space of all continuous functions defined on $K$ and $\| \ . \ \|_{\infty}$ denotes the supremum norm on this space. A {\it uniform algebra} is a $\| \ . \ \|_{\infty}$-closed subalgebra $A \subset C(K)$ endowed with the supremum norm that separates the points of $K$. It is known that the pair $(X, C (K))$ has the BPBp whenever $X$ is an Asplund space \cite[Corollary 2.6]{ACK} and it was extended for the pair $(X, A)$ when $A$ is a uniform algebra \cite[Theorem 3.6]{CGK}. We use the ideas from those results to prove that the pair $(X, A)$ has the BPBpp whenever $X$ is a uniformly smooth Banach space and $A$ is a uniform algebra. 

\begin{theorem} Let $X$ be a uniformly smooth Banach space and $A$ be a uniform algebra. The pair $(X, A)$ has the BPBpp.
\label{DKL8}
\end{theorem}

\begin{proof} Indeed, adapt \cite[Lemma 3.5]{CGK} by using Proposition \ref{DKL1} instead of the Bishop-Phelps-Bollob\'as theorem. Then apply it in \cite[Theorem 3.6]{CGK}. Since every uniformly smooth space is reflexive and every operator from a reflexive space into $A$ is Asplund, the result follows.
\end{proof}

We have the follow consequence.

\begin{corollary} Let $X$ be a uniformly smooth Banach space and let $K$ be a compact Hausdorff topological space. Then the pair $(X, C(K))$ has the BPBpp.
\label{DKL9}
\end{corollary}

Next we study the property on direct sums.

 %% It is surely well-known that the Banach spaces $X_1, \ldots, X_n$ are uniformly smooth if and only if the direct sum $X_1 \oplus \ldots \oplus X_n$ is uniformly smooth as well. By Theorem \ref{DKL1} this means that the pairs $(X_1, \K), \ldots, (X_n, \K)$ satisfy the uniform sBPBP-p if and only the pair $(X_1 \oplus \ldots \oplus X_n, \K)$ so does. In the next result, we treat the vector value case of this by using the results in \cite{ACKLM}.

\begin{proposition} Let $X$ be a uniformly smooth Banach space and let $\{Y_j: \ j \in J \}$ be an arbitrary family of Banach spaces.
\begin{itemize}
\item[(a)] The pairs $\left( X, \left( \bigoplus_{j \in J} Y_j \right)_{\ell_{\infty}} \right)$ and $\left( X, \left( \bigoplus_{j \in J} Y_j \right)_{c_0} \right)$ have the BPBpp if and only if the pair $(X, Y_j)$ satisfies it as well for all $j \in J$.
\item[(b)]  If the pair $\left( X, \left( \bigoplus_{j \in J} Y_j \right)_{\ell_1} \right)$ has the BPBpp, then the pair $(X, Y_j)$  satisfies it as well for all $j \in J$.
\end{itemize}
\label{prop100}
\end{proposition}

\begin{proof} For (a), use \cite[Proposition 2.4]{ACKLM} adapting it for our property. For (b), do the same by using \cite[Proposition 2.7]{ACKLM}.
\end{proof}

We do not know if the converse of Proposition~\ref{prop100} (b) is true even for finite sums. We finish this section by commenting that there are $2$-dimensional uniformly smooth Banach spaces $X$ such that the pair $(X, Y)$ fails the BPBpp for some Banach space $Y$.

\begin{example} It is proved in \cite[Corollary 3.3]{KL} that $2$-dimensional real Banach space $X$ is uniformly convex if and only the pair $(X, Y)$ has the BPBp for all Banach spaces $Y$. Let $X_0$ be a $2$-dimensional Banach space which is  uniformly smooth but not strictly convex. Then, there is a Banach space $Y_0$ such that the pair $(X_0, Y_0)$ fails the BPBp and so it can not satisfy the BPBpp either.
\end{example}

\section{The Bishop-Phelps-Bollob\'as point property for bilinear mappings}

In this section our goal is to provide some results about the Bishop-Phelps-Bollob\'as point property for bilinear mappings.  It is not difficult to see that the BPBpp for bilinear mappings implies the BPBp for bilinear mappings. Note by a routinely change of parameters that we may consider $B \in B(X \times Y, Z)$ with $\|B\| \leq 1$ instead of $\|B\| = 1$ in the Definition \ref{sBPBpbilinear} (we will use this in the Theorem \ref{DKL13}). It is worth to mention that the pair $(\ell_1 \times \ell_1, \K)$ fails the BPBpp for bilinear forms since it fails the BPBp for bilinear forms.

Our first result gives a partial characterization for the pair $(X \times Y, \K)$ to have the BPBpp for bilinear forms. It was proved in \cite[Proposition 2.4]{ABGM} (and independently in \cite[Theorem 1.1]{Dai}) that if $Y$ is a uniformly convex Banach space then the pair $(X \times Y, \K)$ has the BPBp for bilinear forms if and only if the pair $(X, Y^*)$ has the BPBp for operators. We will do the same to our property but assuming now that $Y$ is a Hilbert space. It is easy to check that if the pair $(X \times Y, \K)$ has the BPBp for bilinear forms then the pair $(X, Y^*)$ has the BPBp for operators by using the natural identification between the Banach spaces $B(X \times Y, \K)$ and $\mathcal{L}(X, Y^*)$. The same happens in our case. So we have to prove the converse. That is what we do in the next result.

\begin{theorem} Let $X$ be a uniformly smooth Banach space and let $H$ be a Hilbert space. Then the pair $(X \times H, \K)$ has the BPBpp for bilinear forms if and only if the pair $(X, H^*)$ has the BPBpp (for operators).
\label{DKL10}
\end{theorem}

\begin{proof} Let $\e > 0$ be given. Assume that the pair $(X, H^*)$ has the BPBpp for operators with $\eta(\e) > 0$. Consider $\delta_H(\e) > 0$ the modulus of uniform convexity of $H$. Let $B: X \times H \longrightarrow \K$ be a bilinear form with $\|B\| = 1$ and $(x_0, h_0) \in S_X \times S_H$ be such that
\begin{equation*}
\re B(x_0, h_0) > 1 - \min \left\{ \delta_H(\e), \eta(\delta_H(\e)) \right\}.
\end{equation*}
Define the bounded linear operator $T: X \rightarrow H^*$ by $T(x)(h) := B(x, h)$ for all $x \in X$ and $h \in H$. Then $\|T\| = \|B\| = 1$ and
\begin{equation*}
\|T(x_0)\| \geq \re T(x_0)(h_0) = \re B(x_0, h_0) > 1 - \eta(\delta_H(\e)).
\end{equation*}
There exists $S \in \mathcal{L}(X, H^*)$ with $\|S\| = 1$ such that
\begin{equation*}
\|S(x_0)\| = 1 \ \ \ \mbox{and} \ \ \ \|S - T\| < \delta_H(\e) < 2 \e.
\end{equation*}
Let $h_1 \in S_H$ be such that $\re h_1(S(x_0)) = \re S(x_0)(h_1) = \|S(x_0)\| = 1$. We prove that $\|h_0 - h_1\| < \e$. Note first that since
\begin{equation*}
\delta_H(\e) > \|S - T\| \geq \re T(x_0)(h_0) - \re S(x_0)(h_0) > 1 - \delta_H(\e) - \re S(x_0)(h_0),
\end{equation*}
$\re S(x_0)(h_0) > 1 - 2 \delta_H(\e)$. Then
\begin{equation*}
\left\| \frac{h_0 + h_1}{2} \right\| \geq \re \left( \frac{ S(x_0)(h_0) + S(x_0)(h_1)}{2} \right) > \frac{1 - 2 \delta_H(\e) + 1}{2} = 1 - \delta_H(\e).
\end{equation*}
So $\|h_0 - h_1\| < \e$ as desired. Since $H$ is Hilbert, we can find a linear isometry $R: H \longrightarrow H$ with $\|R\| = 1$ such that $R(h_0) = h_1$ and $\|R - Id_H\| < \e$. We define the bilinear form $A: X \times H \longrightarrow \K$ by $A(x, h) := S(x)(R(h))$ for all $(x, h) \in X \times H$. Then $\|A\| \leq 1$ and $|A(x_0, h_0)| = |S(x_0)(R(h_0))| = |S(x_0)(h_1)| = \re S(x_0)(h_1) = 1$. So $\|A\| = |A(x_0, h_0)| = 1$. Moreover, for all $(x, h) \in S_X \times S_H$, we have that
\begin{eqnarray*}
|A(x, h) - B(x, h)| &\leq& |S(x)(R(h)) - S(x)(h)| + |S(x)(h) - T(x)(h)|  \\
&\leq& \|R - Id_H\| + \|S - T\| \\
&<& 3 \e.
\end{eqnarray*}
Since $(x, h) \in S_X \times S_H$ is arbitrary, we get that $\|A - B\| < 3 \e$. This shows that the pair $(X \times H, \K)$ has the BPBpp for bilinear forms.
\end{proof}

As a consequence of Theorem \ref{DKL10}, we have the following corollary.

\begin{corollary} Let $H_1$ and $H_2$ be Hilbert spaces. Then the pair $(H_1 \times H_2, \K)$ has the BPBpp for bilinear forms.
\label{DKL12}
\end{corollary}
\begin{proof} The pair $(H_1, H_2^*)$ has the BPBpp by Theorem \ref{DKL4}. Hence, Theorem \ref{DKL10} gives the desired result.
\end{proof}

In the following, when the range space has the property $\beta$, we get a positive result as in the operator case.
\begin{theorem}
Let $X, Y$ and $Z$ be Banach spaces. Suppose that the pair $(X \times Y, \K)$ has the BPBpp for bilinear forms and that $Z$ has the property $\beta$. Then the pair $(X \times Y, Z)$ has the BPBpp for bilinear mappings.
\label{DKL13}
\end{theorem}

\begin{proof} Let $Z$ be a Banach space satisfying the property $\beta$ with constant $\rho \in [0, 1)$ and assume that $X$ and $Y$ are Banach spaces. Let $\e > 0$ be given and consider $\xi > 0$ satisfying (\ref{propbeta}) in the Proposition \ref{DKL3}. Let $B \in B(X \times Y, Z)$ with $\|B\|= 1$ and $(x_0, y_0) \in S_X \times S_Y$ be such that 
\begin{equation*}
\|B(x_0, y_0)\| > 1 - \eta(\xi),
\end{equation*}
where $\eta(\e) > 0$ is the constant for the pair $(X \times Y, \K)$ which we are assuming to have the BPBpp. There exists some $\alpha_0 \in \Lambda$ such that
\begin{equation*}
\re (z_{\alpha_0}^* \circ B)(x_0, y_0) = \re z_{\alpha_0}^* (B(x_0, y_0)) > 1 - \eta(\delta).
\end{equation*}
Then there exists $\widetilde{A} \in B(X \times Y, \K)$ with $\|\widetilde{A}\| = 1$ such that 
\begin{equation*}
| \widetilde{A} (x_0, y_0)| = 1 \ \ \mbox{and} \ \ \| \widetilde{A} - (z_{\alpha_0}^* \circ B) \| < \xi.
\end{equation*}
Define $A: X \times Y \longrightarrow Z$ by
\begin{equation*}
A(x, y) := B(x, y) + \left[ \left( 1 + \frac{\e}{4} \right) \widetilde{A}(x, y) - (z_{\alpha_0}^* \circ B) (x, y) \right] z_{\alpha_0}
\end{equation*}
for all $(x, y) \in X \times Y$. Notice that for all $\alpha \in \Lambda$ and $(x, y) \in X \times Y$, we have
\begin{equation*}
z_{\alpha}^* (A(x, y)) = z_{\alpha}^* (B(x, y)) + z_{\alpha}^* (z_{\alpha_0}) \left[ \left( 1 + \frac{\e}{4} \right) \widetilde{A}(x, y) - (z_{\alpha_0}^* \circ B)(x, y) \right].
\end{equation*}
So if $\alpha = \alpha_0$, then $z_{\alpha_0}^*(A(x, y)) = \left(1  + \frac{\e}{4}\right) \widetilde{A} (x, y)$ and this implies that $|z_{\alpha_0}^*(A(x, y))| \leq 1 + \frac{\e}{4}$. On the other hand, if $\alpha \not= \alpha_0$, then
\begin{equation*}
|z_{\alpha}^*(A(x, y))| \leq 1 + \rho \left( \frac{\e}{4} + \xi \right) < \left(1 - \frac{\e}{4} \right)(1 - \xi) < 1 + \frac{\e}{4}.
\end{equation*}
Since $|z_{\alpha_0}^*(A(x_0, y_0))| = 1 + \frac{\e}{4}$, $\|A\| = \|A(x_0, y_0)\|$. Also, we have that $\|B - A\| < \frac{\e}{4} + \xi < \e$. So if $C := \frac{A}{\|A\|}$ then we have that $\|C(x_0, y_0)\| = 1$ and $\|C - B\| < 2 \e$, proving that the pair $(X \times Y, Z)$ has the BPBpp for bilinear mappings.
\end{proof}

As a consequence of Theorem \ref{DKL13}, we have the following corollary.

\begin{corollary} Let $H_1$ and $H_2$ be Hilbert spaces and let $Z$ be a Banach space with the property $\beta$. Then the pair $(H_1 \times H_2, Z)$ has the BPBpp for bilinear mappings.
\end{corollary}

\begin{proof} This is just a combination of Corollary \ref{DKL12} and Theorem \ref{DKL13}.
\end{proof}

Let us now consider compact bilinear mappings. Let $X, Y$ and $Z$ be Banach spaces. We say that the bilinear mapping $B: X \times Y \longrightarrow Z$ is {\it compact} if $B(B_X \times B_Y) \subset Z$ is precompact in $Z$. We denote by $\mathcal{K}(X \times Y, Z)$ the set of all compact bilinear mappings from $X \times Y$ into $Z$. We define the {\it BPBpp for compact bilinear mappings} by using just compact bilinear mappings in Definition \ref{sBPBpbilinear}. Indeed, we consider compact bilinear mappings $A$ and $B$ in that definition.  Our aim in the next lines is to prove that the pair $(H_1 \times H_2, C(K))$ has the BPBpp for compact bilinear mappings whenever $H_1$ and $H_2$ are Hilbert spaces and $K$ is a compact Hausdorff topological space. First, we prove two auxiliary results and our promised one will be a consequence of them.

 For a function $\varphi: K \longrightarrow B (X \times Y, \K)$, we say that $\varphi$ is {\it $\tau_p$-continuous} if the mapping $t \longmapsto \varphi(t)(x, y)$ is continuous on $K$ for each $(x, y) \in X \times Y$. In the next lemma, we prove that there exists a natural (isometric) identification between the spaces $B (X \times Y, C(K))$ and the space of all $\tau_p$-continuous functions from $K$ into $B (X \times Y, \K)$ endowed with the supremum norm $\|\varphi\| = \sup_{t \in K} \|\varphi(t)\|$.

\begin{proposition}{\cite[Theorem 1, p. 490]{DS}} Let $X$ and $Y$ be Banach spaces. Let $K$ be a compact Hausdorff topological space. Then,
	\begin{itemize}
		\item[(i)] there exists an isomorphic isometry between $B (X \times Y, C(K))$ and the set of all $\tau_p$-continuous functions from $K$ into $B (X \times Y, \K)$ and
		\item[(ii)] The subspace $\mathcal{K}(X \times Y, C(K))$ of all compact bilinear mappings from $X \times Y$ into $C(K)$ corresponds to the set of all norm-continuous functions.
	\end{itemize}
	\label{DKL14}
\end{proposition}

\begin{proof} We first prove (i). Let $B \in B (X \times Y, C(K))$ and define $\varphi: K \longrightarrow B (X \times Y, \K)$ by the relation
	\begin{equation}
		\varphi(t)(x, y) := B(x, y)(t) \ \ \left( t \in K \ \mbox{and} \ (x, y) \in X \times Y \right).
		\label{ck1}
	\end{equation}
	Then the function $t \longmapsto \varphi(t)(x, y)$ is continuous on $K$ for each $(x, y) \in X \times Y$ since $B(x, y) \in C(K)$ for each $(x, y) \in X \times Y$. Conversely, if $\varphi: K \longrightarrow B (X \times Y, \K)$ is a $\tau_p$-continuous function, we define $B \in B (X \times Y, C(K))$ as in (\ref{ck1}) and it is not difficult to see that $B$ is a continuous bilinear mapping such that $\|B\| = \|\varphi\|$. 
	
	  Now we prove (ii). Let $B: X \times Y \longrightarrow C(K)$ be a compact bilinear mapping. Consider $\varphi: K \longrightarrow B (X \times Y, \K)$ defined by (\ref{ck1}). We prove that $t \longmapsto \varphi(y)(x, y) = B(x, y)(t)$
	is norm-continuous. Let $(t_{\alpha})_{\alpha} \subset K$ be such that $t_{\alpha} \longrightarrow t_0 \in K$. Then
	\begin{eqnarray*}
		\|\varphi(t_{\alpha}) - \varphi(t)\| &=& \sup_{(x, y) \in B_X \times B_Y} \left| \varphi(t_{\alpha})(x, y) - \varphi(t_0)(x, y) \right| \\
		&=& \sup_{(x, y) \in B_X \times B_Y} \left| B(x, y)(t_{\alpha}) - B(x, y)(t_0) \right| \longrightarrow 0
	\end{eqnarray*}
	since $B(B_X \times B_Y) \subset C(K)$ is equicontinuous and bounded by the Arzel\`{a}-Ascoli theorem for $C(K)$ \cite[Theorem 7, p.266]{DS}. This shows that $t \longmapsto \varphi(t)(x, y)$ is norm-continuous for all $(x, y) \in X \times Y$. On the other hand, let $\varphi: K \longrightarrow B (X \times Y, \K)$ be a norm-continuous function. Define again $B: X \times Y \longrightarrow C(K)$ by the relation (\ref{ck1}). Given $\e > 0$ and $t_0 \in K$, there exists a neighborhood $U_{t_0}$ of $t_0$ such that if $t \in U_{t_0}$, then $\|\varphi(t) - \varphi(t_0)\| < \e$. So
	\begin{eqnarray*}
		\sup_{(x, y) \in B_X \times B_Y} | B(x, y)(t) - B(x, y)(t_0)| &=& \sup_{(x, y) \in B_X \times B_Y} |\varphi(t)(x, y) - \varphi(t_0)(x, y)| \\
		&=& \|\varphi(t) - \varphi(t_0)\| < \e.
	\end{eqnarray*}
	Hence if $t \in U_{t_0}$, then $| B(x, y)(t) - B(x, y)(t_0)| < \e$ for all $(x, y) \in B_X \times B_Y$. This shows that the set $B(B_X \times B_Y)$ is equicontinuous in $C(K)$ and since this set is already bounded, we may conclude that $B$ is a compact bilinear mapping.
\end{proof}

In order to show that the pair $(H_1 \times H_2, C(K))$ has our property for compact bilinear mappings, we will prove first that it is possible to carry the property from the pair $(X \times Y, \K)$ to the pair $(X \times Y, C(K))$ by using Proposition \ref{DKL14}. It is worth to mention that this result is already known for the Bishop-Phelps-Bollob\'as property for compact operators by using a different technique that the one that we are using in here \cite[Theorem 3.15(c)]{DGMM}.

\begin{theorem} Let $X$ and $Y$ be Banach spaces. Let $K$ be a compact Hausdorff topological space. Suppose that the pair $(X \times Y, \K)$ has the BPBpp for bilinear forms. Then the pair $(X \times Y, C(K))$ has the BPBpp for compact bilinear mappings.
	\label{DKL15}
\end{theorem}

\begin{proof} Given $\e > 0$, we consider $\eta(\e) > 0$ the BPBpp constant for the pair $(X \times Y, \K)$. Let $B \in \mathcal{K}(X \times Y, C(K))$ with $\|B\| = 1$ and $(x_0, y_0) \in S_X \times S_Y$ be such that $\|B(x_0, y_0)\|_{\infty} > 1 -\eta \left(\frac{\e}{2}\right)$. Now define $\varphi: K \longrightarrow B (X \times Y, \K)$ by the relation (\ref{ck1}). Since $B$ is compact, $\varphi$ is norm-continuous by using Proposition \ref{DKL14}. Consider $t_0 \in K$ such that
\begin{equation*}
|\varphi(t_0)(x_0, y_0)| = |B(x_0, y_0)(t_0)| > 1 - \eta\left(\frac{\e}{2}\right).
\end{equation*}
Then there is $\widetilde{B} \in B (X \times Y, \K)$ with $\|\widetilde{B}\| = 1$ such that
\begin{equation*}
|\widetilde{B}(x_0, y_0)| = 1 \ \ \ \mbox{and} \ \ \ \|\widetilde{B} - \varphi(t_0)\| < \frac{\e}{2}.
\end{equation*} 
Consider the retraction $r: B (X \times Y, \K) \longrightarrow B_{B (X \times Y, \K)}$ defined for $C \in B (X \times Y, \K)$ by
\begin{equation*}
r(C) := C \ \ \mbox{if} \ \ \|C\| \leq 1 \ \ \mbox{and} \ \ \ r(C) := \frac{1}{\|C\|} C \ \ \mbox{if} \ \ \|C\| \geq 1.
\end{equation*}
Now define the norm-continuous map $\psi: K \longrightarrow B (X \times Y, \K)$ by
\begin{equation*}
\psi(t) := r (\varphi(t) + \widetilde{B} - \varphi(t_0)) \ (t \in K).
\end{equation*}
Then $\psi (t_0) = r(\widetilde{B}) = \widetilde{B}$. Now consider $A: X \times Y \longrightarrow C(K)$ defined by $A(x, y)(t) := \psi(t)(x, y)$ for every $t \in K$ and $(x, y) \in X \times Y$. Then $A \in \mathcal{K}(X \times Y, C(K))$ with $\|A\| \leq 1$ and 
\begin{equation*}
1 \geq \|A\| \geq \|A(x_0, y_0)\|_{\infty} \geq |A(x_0, y_0)(t_0)| = |\psi(t_0)(x_0, y_0)| = |\widetilde{B}(x_0, y_0)| = 1.
\end{equation*}
Then $\|A\| = \|A(x_0, y_0)\|_{\infty} = 1$. It remains to prove that $\|A - B\| < \e$. To prove this, note first note that if $C \in B (X \times Y, \K)$ is such that $1 \leq \|C\| \leq 1 + \frac{\e}{2}$, then
\begin{equation*}
\|r(C) - C\| = \left\| \frac{1}{\|C\|}C - C \right\| = \|C\| - 1 \leq \frac{\e}{2}.
\end{equation*}
and then
\begin{eqnarray*}
\|A - B\| &=& \sup_{t \in K} \|\psi(t) - \varphi(t)\| \\
&=& \sup_{t \in K} \| r(\varphi(t) + \widetilde{B} - \varphi(t_0)) - (\varphi(t) + \widetilde{B} - \varphi(t_0)) + \widetilde{B} - \varphi(t_0)\| \\
&\leq& \frac{\e}{2} + \|\widetilde{B} - \varphi(t_0)\| < \e.
\end{eqnarray*}
This proves that the pair $(X \times Y, C(K))$ has the BPBpp for bilinear mappings.
	
\end{proof}

\begin{corollary} Let $H_1$ and $H_2$ be Hilbert spaces and let $K$ be a compact Hausdorff topological space. Then the pair $(H_1 \times H_2, C(K))$ has the BPBpp for compact bilinear mappings.
	\label{DKL16}
\end{corollary}

\begin{proof} 
	The proof is just a combination of Corollary \ref{DKL12} and Theorem \ref{DKL15}.
\end{proof}

\noindent {\bf Acknowledgement.} This article was written during the visit of the first author to the Dongguk University in Seoul, South Korea. He would like to thank all support and hospitality received there.

\end{document}